\documentclass[11pt,reqno]{amsproc}

\title[Contrast between Lagrangian and Eulerian regularity of Euler equations]{Contrast between Lagrangian and Eulerian analytic regularity properties of Euler equations}
\author{Peter Constantin}
\author{Igor Kukavica}
\author{Vlad Vicol}

\address{Princeton University, Princeton, NJ 08544}
\email{const@math.princeton.edu}

\address{University of Southern California, Los Angeles, CA 90089}
\email{kukavica@usc.edu}

\address{Princeton University, Princeton, NJ 08544}
\email{vvicol@math.princeton.edu}

\usepackage[margin=1.25in]{geometry}
\usepackage{amsmath, amsthm, amssymb}
\usepackage{color}
\usepackage{times}

\newcommand\aaa{Y}

\theoremstyle{plain}
\newtheorem{theorem}{Theorem}[section]

\theoremstyle{definition}
\newtheorem{remark}[theorem]{Remark}

\numberwithin{equation}{section}

\def\RR{{\mathbb R}}
\def\EE{{\mathcal E}}
\def\HH{{\mathcal H}}

\def\TT{{\mathbb T}}
\def\ZZ{{\mathbb Z}}
\def\eps{{\varepsilon}}
\def\phi{{\varphi}}

\def\curl{\mathop{\hbox{curl}} \nolimits}
\def\ddiv{\mathop{\hbox{div}} \nolimits}
\def\supp{\mathop{\hbox{supp}} \nolimits}
\def\arctanh{\mathop{\hbox{arctanh}} \nolimits}

\def\comma{ {,\qquad{}} }            

\begin{document}


\begin{abstract}
We consider the incompressible Euler equations on $\RR^d$ or $\TT^d$, where $d \in \{ 2,3 \}$. We prove that:\\
(a) In Lagrangian coordinates the equations are locally well-posed in spaces with {\em fixed real-analyticity radius} (more generally, a fixed Gevrey-class radius).\\
(b) In Lagrangian coordinates the equations are locally well-posed in {\em highly anisotropic spaces}, e.g.~Gevrey-class regularity in the label $a_1$ and Sobolev regularity in the labels $a_2,\ldots,a_d$.\\
(c) In Eulerian coordinates both results (a) and (b) above are false. 
\hfill {July 13, 2015}.
\end{abstract}




\maketitle

\section{Introduction}
\label{sec:intro}

The Euler equations for ideal incompressible fluids have two formulations, the Eulerian and the Lagrangian one (apparently both due to Euler~\cite{Euler1757}).  In the Eulerian formulation the unknown functions are velocity and pressure, recorded at fixed locations in space. Their time evolution is determined by equating the rates of change of momenta to the forces applied, which in this case are just internal isotropic forces maintaining the incompressible character of the fluid. In the Lagrangian formulation the main unknowns are the particle paths, the trajectories followed by ideal particles labeled by their initial positions. The Eulerian and Lagrangian formulations are equivalent in a smooth regime in which the velocity is in the H\"older class ${\mathcal C}^s$, where $s>1$. The particle paths are just the characteristics associated to the Eulerian velocity fields. 

In recent years it was proved~\cite{Chemin92,Gamblin94,Serfati95b,Sueur11,GlassSueurTakahashi12,Shnirelman12,Nadirashvili13,FrischZheligovsky14a,ZheligovskyFrisch14,ConstantinVicolWu14} that the Lagrangian paths are time-analytic, even in the case in which the Eulerian velocities are only ${\mathcal C}^s$, with $s>1$. In contrast, if we view the Eulerian solution as a function of time with values in ${\mathcal C}^s$, then this function is everywhere discontinuous for generic initial data~\cite{CheskidovShvydkoy10,HimonasAlexandrouMisiolek10,MisiolekYoneda12,MisiolekYoneda14}. This points to a remarkable difference between the Lagrangian and Eulerian behaviors, in the not-too-smooth regime.

In this paper we describe a simple but astonishing difference of behaviors in the analytic regime: The radius of analyticity is locally in time conserved in the Lagrangian formulation (Theorem~\ref{thm:isotropic}), but may deteriorate instantaneously in the Eulerian one (Remark~\ref{rem:analyticity}). Moreover, the Lagrangian formulation allows solvability in anisotropic classes, e.g.~functions which have analyticity in one variable, but are not analytic in the others (Theorem~\ref{thm:anisotropic}). In contrast, the Eulerian formulation is ill-posed in such functions spaces (Theorem~\ref{thm:counterexample}).

\subsection{Velocity in Lagrangian coordinates}

We consider the Cauchy problem for the incompressible homogeneous Euler equations
\begin{align}
u_t + u\cdot \nabla u+ \nabla p&=0 \label{eq:E:1}\\
\nabla \cdot u&=0 \label{eq:E:2}\\
u(x,0) &=u_0(x) \label{eq:E:3}
\end{align}
where $(x,t) \in \RR^d \times [0,\infty)$, and $d\in \{2,3\}$. 
In order to state our main results, we first rewrite the Euler equations in Lagrangian coordinates. Define the particle flow map $X$ by
\begin{align}
\partial_t X(a,t) &= u(X(a,t),t) \label{eq:eta:1}\\ 
X(a,0) &= a \label{eq:eta:2}
\end{align}
where $t\geq 0$, and $a\in \RR^d$ is the Lagrangian label.
The Lagrangian velocity $v$ and the pressure $q$ are obtained by composing with $X$, i.e., 
\begin{align*} 
v(a,t) &= u(X(a,t),t) \\
q(a,t) &= p(X(a,t),t).
\end{align*}
The Lagrangian formulation of the Euler equations \eqref{eq:E:1}--\eqref{eq:E:3} is given in components by
\begin{align}
\partial_t v^{i}  + \aaa^{k}_{i}  \partial_{k} q   &= 0 \label{eq:EL:1}
\comma i=1,\ldots,d
\\
\aaa^{k}_{i} \partial_{k}   v^{i}   &= 0  \label{eq:EL:2}
\end{align}
where we have used the summation convention on repeated indices. The derivatives $\partial_k$ are with respect to the label direction $a_k$ and  $\aaa_{i}^{k}$ represents the $(k,i)$ entry of the matrix inverse of the Jacobian of the particle map, i.e.,
\begin{align*} 
\aaa(a,t) = (\nabla_a X(a,t))^{-1}.
\end{align*}
We henceforth drop the index $a$ on gradients, as it will be clear from the context when the gradients are taken with respect to Lagrangian variables $a$ or with respect to the Eulerian variable $x$.
From \eqref{eq:E:2} it follows that $\det(\nabla X) = 1$, and thus, differentiating $\partial_t X = v$ with respect to labels, and inverting the resulting matrix, we obtain 
\begin{align} 
\aaa_{t} = - \aaa (\nabla v ) \aaa.  \label{eq:EL:3}
\end{align}
The closed system for $(v,q,\aaa)$ is supplemented with the initial conditions 
\begin{align*}
v(a,0) &= v_0(a) = u_0(a) \\
\aaa(a,0) &= I
\end{align*}
where $I$ is the identity matrix. In the smooth category, the Lagrangian equations \eqref{eq:EL:1}--\eqref{eq:EL:3} are equivalent to the Eulerian ones \eqref{eq:E:1}--\eqref{eq:E:3}.

\subsection{Vorticity in Lagrangian coordinates}
For $d=2$ the Eulerian scalar vorticity $\omega = \nabla^\perp \cdot u$ is conserved along particle trajectories, that is, the Lagrangian vorticity 
\begin{align*} 
\zeta (a,t) = \omega (X(a,t),t) 
\end{align*}
obeys 
\begin{align}
\zeta (a,t) = \omega_0(a) \label{eq:vort:2D}
\end{align}
for $t\geq 0$. The Lagrangian velocity $v$ may then be computed
from the Lagrangian vorticity $\zeta$
using the elliptic curl-div system
\begin{align}
\eps_{ij} \aaa_i^k \partial_k v^j &=  \aaa^{k}_{1} \partial_{k} v^{2} - \aaa^{k}_{2} \partial_{k} v^{1} = \zeta = \omega_0 \label{eq:vort:2D:1}\\
\aaa_i^k \partial_k v^i &= \aaa^{k}_{1} \partial_{k} v^{1} + \aaa^{k}_{2} \partial_{k} v^{2} = 0\label{eq:vort:2D:2}
\end{align}
where $\eps_{ij}$ is the sign of the permutation $(1,2) \mapsto (i,j)$.
The equation \eqref{eq:vort:2D:1} above represents the conservation of the Lagrangian vorticity, while \eqref{eq:vort:2D:2} stands for the Lagrangian divergence-free condition. Note that the right sides of \eqref{eq:vort:2D:1}--\eqref{eq:vort:2D:2} are {time independent}.

For $d=3$ the Eulerian vorticity vector $\omega = \nabla \times u$ is not conserved along particle trajectories, and the replacement of \eqref{eq:vort:2D} is the vorticity transport formula
\begin{align} 
\zeta^i (a,t) = \partial_k X^i(a,t) \omega_0^k(a) \label{eq:vort:3D}.
\end{align}
Thus, in three dimensions, the elliptic curl-div system becomes
\begin{align} 
\eps_{ijk}  \aaa_j^l \partial_l v^k &= \zeta^i = \partial_k X^i \omega_0^k \label{eq:vort:3D:bad}\\
\aaa_i^k \partial_k v^i &=0\label{eq:vort:3D:2}
\end{align}
where $\epsilon_{ijk}$ denotes the standard antisymmetric tensor.
In order to make use of the identity \eqref{eq:vort:3D:bad}, we need to reformulate it so that the right side is time-independent, in analogy to the two-dimensional case. 
Multiplying \eqref{eq:vort:3D:bad}
with 
$Y_{i}^{m}$ and summing in $i$, we get
  \begin{align} 
   \eps_{ijk}  \aaa^m_i \aaa_j^l \partial_l v^k &= \omega_0^m \label{eq:vort:3D:1}
   \comma m=1,2,3,
  \end{align}
which is a form of the Cauchy identity
containing only $\aaa$.
Recall here 
the standard Cauchy invariants~\cite{Cauchy1827,ZheligovskyFrisch14}
\begin{align} 
\eps_{ijk} \partial_j v^l \partial_k X^l = \omega_0^i
   \comma i=1,2,3,
\label{eq:Cauchy}
\end{align}
which can be obtained by taking the Lagrangian curl of the Weber
formula~\cite{Weber68,Constantin01a}. 
Thus, for $d=3$ we solve \eqref{eq:vort:3D:2} and \eqref{eq:vort:3D:1} for $\nabla v$ in terms of $\aaa$ and $\omega_0$. Note that, as in the $d=2$ case, this system has a right side which is time independent.

\subsection{Isotropic and anisotropic Lagrangian Gevrey spaces}
First we recall the definition of the Gevrey spaces. Fix $r > d/2$, so that $H^r(\RR^d)$ is an algebra (we may replace $H^r(\RR^d)$ with $W^{r,p}(\RR^d)$ for $r > d/p$ and $p \in (1,\infty))$. 
For a Gevrey-index $s\geq1$ and Gevrey-radius $\delta>0$, we denote the isotropic Gevrey norm by
\begin{align}
\|f\|_{G_{s,\delta}} = \sum_{\beta \geq 0} \frac{\delta^{|\beta|}}{|\beta|!^s} \| \partial^\beta f\|_{H^r} 
= \sum_{m\geq 0}   \frac{\delta^m }{m!^s} \sum_{|\beta| = m} \| \partial^\beta f\|_{H^r} 
\label{eq:G:s:delta}
\end{align}
where $\beta \in {\mathbb N}_0^d$ is a multi-index.
Also, let
$G_{s,\delta}$ be the set of functions for which the above norm
is finite.
When $s=1$ this set consists of analytic functions extendable 
analytically to the strip of radius $\delta$, and which are bounded uniformly in this strip (the latter property is encoded in the summability property of the norm).

Similarly, given a coordinate  $j \in \{ 1,\ldots,d\}$, we define the anisotropic $s$-Gevrey norm with radius $\delta>0$  by 
\begin{align*}
\|f\|_{G_{s,\delta}^{(j)}} 
= \sum_{m\geq 0}   \frac{\delta^m}{ m!^s}   \| \partial^m_j f\|_{H^r} 
\end{align*}
that is, among all multi-indices $\beta$ with $|\beta|=m$, we only consider $\beta = (\beta_k)$ with $\beta_k= m \delta_{jk}$, where $\delta_{jk}$ is as usual the Kronecker symbol.

\subsection{Main results}
We have the following statement asserting persistence of the Gevrey radius for solutions of the Lagrangian Euler equation. 

\begin{theorem}[\bf Persistence of the Lagrangian Gevrey radius] \label{thm:isotropic}
Assume that $v_0 \in L^2$ and
\begin{align*}
\nabla v_0  \in G_{s,\delta}
\end{align*}
for some Gevrey-index $s\geq 1$ and a Gevrey-radius $\delta>0$.
Then there exists $T>0$ and a unique solution $ (v,\aaa) \in C([0,T];H^{r+1}) \times C([0,T],H^{r})$ of the Lagrangian Euler system \eqref{eq:EL:1}--\eqref{eq:EL:3}, which moreover satisfies
\begin{align*}
\nabla v, \aaa \in L^{\infty}([0,T], G_{s,\delta}).
\end{align*}
\end{theorem}

On the other hand, if the uniform analyticity radius of the solution $u(x,t)$ of \eqref{eq:E:1}--\eqref{eq:E:3} is measured with respect to the Eulerian coordinate $x$, then this radius is in general not conserved in time, as may be seen in the following example.

\begin{remark}[\bf Decay of the Eulerian analyticity radius] \label{rem:analyticity}
We recall from Remark 1.3 in~\cite{KukavicaVicol11b} that  there exist solutions to \eqref{eq:E:1}--\eqref{eq:E:3} whose Eulerian real-analyticity radius decays in time. Consider the explicit shear flow example (cf.~\cite{DiPernaMajda87a,BardosTiti10}) given by
\begin{align}\label{eq:DiPernaMajda}
  u(x,t) = ( f(x_2), 0 , g(x_1 - t f(x_2)) )
\end{align}
which satisfies \eqref{eq:E:1}--\eqref{eq:E:2} with vanishing pressure in $d=3$, for smooth $f$ and $g$.
For $s=1$ we may for simplicity consider the domain to be the periodic box $[0,2\pi]^3$, and
{let} 
\[
f(y) = \sin(y) \qquad \mbox{and} \qquad 
g(y) = \frac{1}{\sinh^2(1) + \sin^2(y)}.
\]
It is easily verified that the uniform in $x_1$ and $x_2$ real-analyticity radius of $u(x,t)$ decays as 
\[
\frac{1}{t+1}
\] 
for all $t>0$, and is thus not conserved. 
Note however that the above example does not provide the necessary counterexample to Theorem~\ref{thm:isotropic},
since $g$ does not belong to the periodic version of $G_{1,1}$. Indeed, $(-1)^n g^{(2n)}(0) \geq (2n!)/4$, and thus the series defining $\|g\|_{G_{1,1}}$ is not summable.
\end{remark}

The next statement shows indeed that  Theorem~\ref{thm:isotropic} does not hold in the Eulerian setting.

\begin{theorem}\label{thm:example}
There exists a smooth periodic 
divergence-free function $u_0$
such that 
\begin{align}
\| u_0\|_{G_{1,1}} < \infty \label{eq:IC}
\end{align}
and such that 
\begin{align}
\| u(t) \|_{G_{1,1}}  = \infty \label{eq:t>0}
\end{align}
for any $t>0$.
\end{theorem}

The example proving Theorem~\ref{thm:example} is
provided in Section~\ref{sec:G1delta}.

We are indebted to A.~Shnirelman~\cite{Shnirelman15} for the following example, pointing out that the results of Theorem~\ref{thm:isotropic} are sharp in the sense that the time of persistence $T$ of the Lagrangian analyticity radius may be strictly less than the maximal time of existence of a real-analytic solution.
\begin{remark}[\bf Time of persistence of the Lagrangian analyticity radius~\cite{Shnirelman15}]
Consider the stationary solution 
\[
u(x_1,x_2) = (\sin x_1 \cos x_2, - \cos x_1 \sin x_2)
\] 
of the Euler equations in $\RR^2$. This is an entire function of $(x_1,x_2)$, and moreover, the $x_1$-axis is invariant under the induced dynamics.  
Abusing notation we denote by 
\[
X_1(a_1,t) = X_1(a_1,0,t)
\] 
the image of the point $(a_1,0)$ under the flow map at the moment $t$, and by 
\[
Y_{22}(a_1,t) = Y_{22}(a_1,0,t) = (\partial_{a_1} X_1)(a_1,0,t)
\] 
its Lagrangian tangential derivative. These functions satisfy the ODE
\begin{align*}
&\frac{d}{dt} X_1(a_1,t) = \sin (X_1(a_1,t)), \qquad X_1(a_1,0) = a_1, \\
&\frac{d}{dt} Y_{22}(a_1,t) = \cos (X_1(a_1,t)) Y_{22}(a_1,t), \qquad Y_{22}(a_1,0) = 1.
\end{align*}
The solution $X_1$ is given by
\[
\cos X_1(a_1,t) = \frac{ (e^{2t}+1) \cos (a_1) - (e^{2t}-1)}{(e^{2t}+1) - (e^{2t}-1) \cos(a_1)}
\]
and its tangential gradient obeys
\[
Y_{22}(a_1,t) = \frac{2 e^t}{(e^{2t}+1) - (e^{2t} -1)\cos(a_1)}.
\]
Thus, for any fixed $t>0$, the function $Y_{22}(a_1,t)$ has a singularity at the complex point $a_1 = \Re a_1 + \Im a_1$ (and its conjugate) satisfying 
\[
\cos(a_1) = \cos (\Re a_1 + i \Im a_1) = \frac{e^{2t}+1}{e^{2t}-1} 
\]
so that
\[
\Re a_1 = 0, \qquad \mbox{and} \qquad |\Im a_1| =  \ln\left( \frac{e^{t}+1}{e^{t}-1}\right).
\]
Note however that this singularity obeys
\begin{align}
|a_1| \to \infty \qquad \mbox{as} \qquad t \to 0^+.
\end{align}
In summary, at any fixed $t>0$ the function $Y(a,t)$ is not anymore entire with respect to the label $a$.
Given any $\delta>0$, we have $\nabla_a v_0 = \nabla_x u_0  \in G_{1,\delta}$, and while $\nabla_x u(\cdot,t) = \nabla_x u_0 \in G_{1,\delta}$ for all $t>0$, there exists 
\[
T = T(\delta) = \ln \left( \frac{e^{\delta}+1}{e^{\delta}-1} \right)>0
\] 
such that $Y(\cdot,t)$, and thus also $\nabla_a v(\cdot,t)$, obey
\[
\|Y(\cdot,t)\|_{G_{1,\delta}},\|\nabla v(\cdot,t)\|_{G_{1,\delta}} \to \infty \qquad \mbox{as} \qquad t \to T(\delta)^-.
\]
Thus, the time of analyticity radius persistence $T$ guaranteed by Theorem~\ref{thm:isotropic} cannot be taken as infinite. Yet, Theorem~\ref{thm:isotropic}  is consistent with $T(\delta) \to 0$ as $\delta \to \infty$.
\end{remark}

The  proof of Theorem~\ref{thm:isotropic} may be used to obtain the local existence and the persistence of the radius for anisotropic Gevrey spaces as well.
\begin{theorem}[\bf Solvability in  Lagrangian anisotropic Gevrey spaces] \label{thm:anisotropic}
For a fixed direction $j\in\{1,\ldots,d\}$, assume that $v_0\in H^{r+1}$ and that
\begin{align*}
\nabla v_0  \in G_{s,\delta}^{(j)}
\end{align*}
for some index $s\geq1$ and radius $\delta>0$.
Then there exists $T>0$ and a unique solution 
$(v,Y)\in C([0,T],H^{r+1})\times C([0,T],H^{r})$ 
of the Lagrangian Euler system \eqref{eq:EL:1}--\eqref{eq:EL:3}, which moreover satisfies
\begin{align*}
\nabla v, \aaa \in L^{\infty}([0,T], G_{s,\delta}^{(j)}).
\end{align*}
\end{theorem}

The above theorem does not hold in the Eulerian coordinates as shown by the next result. The fact that the Eulerian version of the theorem does not hold might not surprise, due to the isotropy and time-reversibility of the Euler equations. On the other hand, the fact that the Lagrangian formulation keeps the memory of initial anisotropy is puzzling.

\begin{theorem}[\bf Ill-posedness in  Eulerian anisotropic real-analytic spaces] \label{thm:counterexample}
There exists $T>0$ and an initial datum $u_0 \in C^\infty(\RR^2)$ for which $u_0$ and $\omega_0$ are  real-analytic in $x_1$, uniformly with respect to $x_2$, such that the unique $C([0,T];H^r)$ solution $\omega(t)$ of the Cauchy problem for the Euler equations \eqref{eq:E:1}--\eqref{eq:E:3} is not real-analytic in $x_1$, for  any $t \in (0,T]$.
\end{theorem}

\section{Ill-posedness in Eulerian anisotropic real-analytic spaces} \label{sec:counterexample}
In this section we prove Theorem~\ref{thm:counterexample}. 
Here, all the derivatives are taken with respect to the Eulerian variables. The idea of the proof is as follows. 
We consider an initial vorticity that is supported in a horizontal strip around the $x_1$ axis and which is nonzero in a horizontal strip and is very highly concentrated near the origin. 
We can construct it such that it is real analytic in $x_1$, but is obviously not real analytic in $x_2$. 
Given that the vorticity is approximately a point vortex at the origin, the corresponding velocity is approximately a pure rotation. 
Then for short time, the Euler equations will evolve in such manner that the vorticity is supported in a slightly deformed but rotated strip. 
The rotation uncovers some of the points that were on the boundary of the original strip, making them points of vanishing vorticity, while covering others. 
Thus, on a horizontal line parallel with the $x_1$ axis, the vorticity instantly acquires an interval on which it must vanish, while it is not identically zero, and hence it cannot possibly continue to be real analytic with respect to $x_1$.

In the detailed proof we first construct a function
  \begin{align*}
   u(x_1,x_2)
   \in C^{\infty}({\mathbb R}^2)
  \end{align*}
such that the following properties hold:
\begin{itemize}
\item[(i)] $\ddiv u=0$ on ${\mathbb R}^2$, $\curl u = \omega$,
\item[(ii)] 
$\supp\, \omega
\subseteq \{(y_1,y_2):-1\leq y_2\leq 1\}$
\item[(iii)] $u_2(1,1)>0$ and $u_2(-1,1)<0$
\item[(iv)] There exists $\eps\in(0,1/2)$ such that
  \begin{align*}
   \omega(x_1,x_2)\neq 0   
   \comma 
    (x_1,x_2)
    \in
    \bigl\{
     (y_1,y_2)
     :
     |y_1-1|<\eps,
     1-\eps<y_2<1
    \bigr\}
  \end{align*}
\item[(v)] (tangential analyticity for $u$) There exists
constants $M_0, \delta_0>0$ such that
  \begin{align}
   |\partial_{1}^{m}u(x_1,x_2)|
   \le
   \frac{M_0 m!}{\delta_0^{m}}
   \label{EQ03}
  \end{align}
with
  \begin{align}
   |\partial_{1}^{m}\partial_{1}u(x_1,x_2)|
   \le
   \frac{M_0 m!}{\delta_0^{m}}
   \label{EQ04}
  \end{align}
and
  \begin{align}
   |\partial_{1}^{m}\partial_{2}u(x_1,x_2)|
   \le
   \frac{M_0 m!}{\delta_0^{m}}
   \label{EQ05}
  \end{align}
\item[(vi)] $\partial^{\alpha}\omega$ converges to $0$
exponentially fast and uniformly 
as $x_1\to\pm\infty$,
uniformly in $x_2$.
\end{itemize}

In order to simplify the presentation, we introduce the following
notation: If $\omega$ is a function (or a measure) with a sufficient
decay at infinity, denote
  \begin{align*}
   u(\omega)
   =
   \int_{{\mathbb R}^2}
           K(x-y)\omega(y)\,dy
  \end{align*}
where
  \begin{align*}
   K(x)
   =\frac{1}{2\pi}
     \left(
      -\frac{x_2}{|x|^2},
      \frac{x_1}{|x|^2}
     \right)
  \end{align*}
denotes the Biot-Savart kernel.
Now, choose a test function
  \begin{align*}
   \psi\in C_{0}^{\infty}({\mathbb R})
  \end{align*}
with values in $[0,1]$ such that
$\int\psi=1$ with $\psi(x)>0$ for $x\in(-1,1)$ and $\psi=1$ on $[-1/4,1/4]$.
Consider the sequence of vorticities
  \begin{align}
   \omega^{(k)}(x_1,x_2)
   =c_0 k^2 \exp\bigl(
          -k^2(x_1^2+x_2^2)
         \bigr)
     \psi(x_2)
   \label{EQ09}
  \end{align}
for $k=1,2,\ldots$, where $c_0$ is a normalizing constant
such that
  \begin{align*}
   \int \omega^{(k)}(x) \,dx\to 1
   \hbox{~~as~} k\to\infty
  \end{align*}
Denote by
  \begin{align*}
   u^{(k)}(x_1,x_2)
   = u(\omega^{(k)}(x_1,x_2))  
   \comma 
   k=1,2,\ldots
  \end{align*}
the corresponding velocities.
Each individual member of this sequence of velocities satisfies the
assumptions
(i), (ii), (iv), (v), and (vi).
(Note however that the constants in \eqref{EQ03}--\eqref{EQ05} depend
on $k$.)
The construction of a desired vorticity is
complete once we show that for $k$
large enough, we have
  \begin{align*}
   u^{(k)}_2(1,1)>0
  \end{align*}
and
  \begin{align*}
   u^{(k)}_2(-1,1)<0.
  \end{align*}
These inequalities for $k$ sufficiently large
indeed follow immediately once we observe that the
sequence \eqref{EQ09} is an approximation of identity, i.e.,
it converges to the Dirac mass $\delta_0$, while the 
velocity
  \begin{align*}
   u=u(\delta_0)
   =\frac{1}{2\pi}
     \left(
      -\frac{x_2}{|x|^2},
      \frac{x_1}{|x|^2}
     \right)
   \end{align*}
corresponding to $\delta_0$ satisfies (iii).
Thus the construction of a velocity satisfying
the properties (i)--(vi) is complete. Denote this velocity
by $u_0$ and the corresponding vorticity
$\omega_0=\curl u_0$.
Now, consider the Euler equation
  \begin{align*}
   \omega_t
   + u(\omega)\cdot\nabla \omega =0
  \end{align*}
with
  \begin{align*}
   \omega(0)=\omega_0
  \end{align*}
where, recall, $u(\omega)$ denotes the velocity computed
from the vorticity $\omega$ via the Biot-Savart law.
By the well-known properties of the Euler equation, the
solution is smooth for all $t>0$.
By (ii) and (iii) and using the Lagrangian variables to
solve the Euler equation, there exists
$t_0>0$ with the following property:
For every $t\in(0,t_0)$, there exists a constant
$\eps_1(t)>0$ such that
  \begin{align}
   \omega(x_1,x_2)=0
   \comma |(x_1,x_2)-(-1,1)|<\eps_1(t)
   \label{EQ17}
  \end{align}
On the other hand, by
(iii) and (iv), we obtain, by possibly reducing $t_0$, that
for every $t\in(0,t_0)$ there exists a constant $\eps_2(t)>0$
such that
  \begin{align}
   \omega(x_1,x_2)
   \neq 0
   \comma |(x_1,x_2)-(1,1)|<\eps_2(t)
   .
   \label{EQ18}
  \end{align}
The properties \eqref{EQ17} and \eqref{EQ18}
contradict the tangential analyticity of $\omega(t)$
at $x_2=1$ for all $t\in(0,t_0)$.

\section{Local solvability in Lagrangian anisotropic Gevrey spaces}
\label{sec:anisotropic}

In this section we prove Theorem~\ref{thm:anisotropic}. For simplicity of the presentation, we give here the proof for $d=2$. 
The proof carries over mutatis mutandis to $d=3$, where the only change arises from using \eqref{eq:vort:3D:1} instead of \eqref{eq:vort:2D:1}. 
These details may be seen in Section~\ref{sec:isotropic}, where the well-posedness 
(by which we mean the existence and uniqueness)
in $3d$ isotropic Gevrey spaces is proven. 

Fix $s \geq1$. Without loss of generality, the direction $j \in \{1,2\}$ may be taken to be $j=1$. 
Fix $\delta >0$ so that $\nabla v_0 \in G_{s,\delta}^{(1)}$ with the norm $M$, that is, the quantity
\begin{align*}
\Omega_m= \| \partial_1^m \nabla v_0\|_{H^r} 
\end{align*}
obeys
\begin{align}
\sum_{m\geq 0} \Omega_m \frac{\delta^m}{m!^s} \leq M
\label{eq:aniso:data}
\end{align}
Recall that $\aaa_0 = I$.

Fix $T>0$, to be chosen further below sufficiently small in terms of $M$, $s$,
and $\delta$. For $m\geq 0$ we define 
\begin{align}
V_m &= V_m(T) = \sup_{t\in [0,T]} \| \partial_1^m \nabla v(t)\|_{H^r} \label{eq:Vm:def}, \\
Z_m &= Z_m(T) = \sup_{t\in [0,T]} t^{-1/2} \| \partial_1^m (\aaa(t)-I) \|_{H^r} \label{eq:Zm:def}.
\end{align}
Observe that in the norm \eqref{eq:Vm:def} the velocity $v$ does not appear without a gradient. Also, we note that the power $-1/2$ of $t$ appearing in \eqref{eq:Zm:def} is arbitrary, in the sense that the proof works with any power in $(-1,0)$.

First we bound $\nabla v$ from the approximate curl-div system \eqref{eq:vort:2D:1}--\eqref{eq:vort:2D:2}, in terms of $\aaa$ and $\omega_0$.
Since $\partial_1^m$ commutes with $\curl$ and $\ddiv$, we may use the Helmholtz decomposition to estimate
\begin{align*}
\Vert  \partial_{1}^{m}\nabla v\Vert_{H^{r}}
\leq C\Vert \partial_1^{m} \curl  v\Vert_{H^{r}} + C \Vert \partial_1^{m}  \ddiv v\Vert_{H^{r}}.
\end{align*}
Further, by appealing to \eqref{eq:vort:2D:1}--\eqref{eq:vort:2D:2}, the Leibniz rule, and the fact that $H^r$ is an algebra, we obtain
\begin{align*}
\Vert  \partial_{1}^{m}\nabla v\Vert_{H^{r}}
&\leq C\Vert \partial_1^{m} (\omega_0 + \eps_{ij} (\delta_{ik} - \aaa_i^k) \partial_k v^j) \Vert_{H^{r}} 
+ C \Vert \partial_1^{m}  ((\delta_{ik} -\aaa_i^k) \partial_k v^i)\Vert_{H^{r}} \notag\\
&\leq C \Vert \partial_{1}^{m}\omega_0\Vert_{H^{r}} 
+C \Vert \aaa-I\Vert_{H^{r}} \Vert \partial_{1}^{m}\nabla v\Vert_{H^{r}}
+ C \Vert \partial_{1}^{m}(\aaa-I)\Vert_{H^{r}} \|\nabla v\|_{H^r} \notag\\
&\qquad + C \sum_{j=1}^{m-1} {m \choose j} \| \partial_1^j (\aaa-I) \|_{H^r} \|\partial_1^{m-j} \nabla v\|_{H^r}.
\end{align*}
Taking a supremum over $t \in [0,T]$ and using the notation \eqref{eq:Vm:def}--\eqref{eq:Zm:def}, we obtain 
\begin{align}
V_m \leq C \Omega_m + C T^{1/2}  Z_0 V_m  + C  T^{1/2} Z_m V_0 + C  T^{1/2}  \sum_{j=1}^{m-1} {m \choose j} Z_j V_{m-j}
\label{eq:Vm}
\end{align}
for all $m\in{\mathbb N}$, while for $m=0$ we have
\begin{align}
V_0 \leq C \Omega_0 + C T^{1/2} Z_0 V_0.
\label{eq:V0}
\end{align}
Note that we have not used here the evolution equation \eqref{eq:EL:1}
for $v$, and have instead appealed to the Lagrangian vorticity conservation~\eqref{eq:vort:2D:1}.

In order to estimate $Z_m$, we use the Lagrangian evolution \eqref{eq:EL:3} in integrated form, and obtain
\begin{align}
I - \aaa(t)  &= \int_0^t \aaa \colon \nabla v \colon \aaa \; d\tau\notag\\
&= \int_0^t (\aaa-I) \colon \nabla v \colon (\aaa-I) \; d\tau + \int_0^t (\aaa-I) \colon \nabla v \; d\tau  \notag\\
&\quad + \int_0^t  \nabla v \colon (\aaa-I)  \; d\tau + \int_0^t   \nabla v   \; d\tau
\label{eq:a:integrated}
\end{align}
for all $t \in [0,T]$. Dividing by $t^{1/2}$ and taking a supremum over $t \in [0,T]$ it immediately follows from \eqref{eq:a:integrated} that
\begin{align}
Z_0 \leq C T^{1/2} (1+T^{1/2} Z_0)^2 V_0.
\label{eq:Z0}
\end{align}
Differentiating \eqref{eq:a:integrated} $m$ times with respect to the label $a_1$, using the Leibniz rule, and the fact that $H^r$ is an algebra, we arrive at
\begin{align*}
&\| \partial_1^m(\aaa(t) -I) \|_{H^r}\notag\\
&\qquad\leq \sum_{|(j,k)|\leq m} \int_{0}^{t} {m\choose j~k} \| \partial_1^j (\aaa-I)\|_{H^r} \|\partial_1^k (\aaa-I) \|_{H^r} \|\partial_1^{m-j-k} \nabla v\|_{H^r} \; d\tau \notag\\
&\qquad\qquad + 2 \sum_{j=0}^m \int_0^t {m \choose j} \|\partial_1^j (\aaa-I) \|_{H^r} \|\partial_{1}^{m-j} \nabla v\|_{H^r} \; d\tau  + \int_0^t \|\partial_1^m \nabla v\|_{H^r} \; d\tau
\end{align*}
for all $m\geq 1$. 
Further, dividing by $t^{1/2}$, taking a supremum over $t \in [0,T]$ and using the notation \eqref{eq:Vm:def}--\eqref{eq:Zm:def}, we obtain 
\begin{align}
Z_m 
&\leq C T^{3/2} \sum_{|(j,k)|\leq m} {m\choose j~k} Z_j Z_k V_{m-j-k} + C T  \sum_{j=0}^m  {m \choose j} Z_j V_{m-j} + C T^{1/2} V_m \notag\\
&\leq C T^{1/2} (T Z_0^2 V_m + T Z_m Z_0 V_0 + T^{1/2} Z_0 V_m + T^{1/2} Z_m V_0 + V_m)  \notag\\
&\qquad + C T^{3/2} \sum_{0 < |(j,k)| < m} {m\choose j~k} Z_j Z_k V_{m-j-k} + C T  \sum_{j=1}^{m-1}  {m \choose j} Z_j V_{m-j}
\label{eq:Zm}
\end{align}
for some constant $C>0$. 

From \eqref{eq:V0} and \eqref{eq:Z0} we obtain that for any $t \in (0,T]$ we have
\begin{align*}
V_0(t) &\leq C_0 \Omega_0 + C_0 t^{1/2} Z_0(t) V_0(t)\\
Z_0(t) &\leq C_0 t^{1/2} \sup_{\tau \in [0,t)} \left( V_0(\tau) (1+t^{1/2} Z_0(\tau))^2 \right)
\end{align*}
for some constant $C_0>0$, 
while the initial data obey
\begin{align*}
V_0(0) &= \|\nabla v_0\|_{H^r} = \Omega_0 \leq M\\
Z_0(0) &= 0.
\end{align*}
Here we used that in view of \eqref{eq:a:integrated}, as long as $\nabla v$ and $\aaa$ are bounded in time, we have  $t^{-1/2}(\aaa(t)-I) \approx t^{1/2} \to 0$  as $t\to 0$.
By the {continuity in time} of $V_0(t)$ and $Z_0(t)$, it follows that there exists 
\begin{align*}
T_1 = T_1(M) > 0
\end{align*}
such that 
\begin{align}
\sup_{t\in[0,T_1]} V_0(t) &\leq  3 C_0 M \label{eq:local:bnd:1} \\
\sup_{t\in[0,T_1]}  Z_0(t) &\leq \frac 12 .\label{eq:local:bnd:2}
\end{align}
This is a time of local existence in $H^r(\RR^d)$ for $\nabla v$ and $a$.

At this stage, we assume that $T$ obeys
\begin{align}
T\leq T_1
\label{eq:T1}
\end{align}
and  define
\begin{align*}
B_m = V_m + Z_m = \sup_{t\in [0,T]} \left( V_m(t) + Z_m(t) \right)
\end{align*}
for all $m\geq 0$. By \eqref{eq:local:bnd:1}--\eqref{eq:local:bnd:2} we have 
\begin{align}
B_0 \leq 3 C_0 M + \frac 12.
\label{eq:B0}
\end{align}
Adding \eqref{eq:Vm} and \eqref{eq:Zm} we arrive at
\begin{align}
B_m &\leq C_1 \Omega_m + C_1 T^{1/2}  (1+ B_0 + T^{1/2} B_0 + T B_0^2) B_m \notag\\
&\qquad + C_1 T^{1/2} (1+T^{1/2}) \sum_{0 < j<m} {m \choose j} B_j B_{m-j} \notag\\
&\qquad + C_1 T^{3/2}  \sum_{0< |(j,k)| < m} {m\choose j~k} B_j B_k B_{m-j-k}
\label{eq:Bm:1}
\end{align} 
for all $m\geq 1$, for some positive constant $C_1\geq 1$.
In view of \eqref{eq:B0} we may take 
\[
0 < T = T(B_0) = T(M)< T_1 
\]
sufficiently small, such that 
\begin{align}
C_1 T^{1/2}  (1+ B_0 + T^{1/2} B_0 + T B_0^2)  \leq \frac 12.
\label{eq:T2}
\end{align}
We thus obtain from \eqref{eq:Bm:1} and \eqref{eq:T2} that
\begin{align}
B_m \leq 2 C_1 \Omega_m 
&+ 2 C_1 T^{1/2} (1+T^{1/2})   \sum_{0 < j < m} {m \choose j} B_j B_{m-j} \notag\\
&+  2 C_1 T^{3/2}  \sum_{0< |(j,k)| < m} {m\choose j~k} B_j B_k B_{m-j-k}
\label{eq:Bm:2}
\end{align}
for all $m \geq 1$.

Finally, denote
\begin{align*}
\|(\nabla v, \aaa-I)\|_{\delta,s,T} = \sum_{m\geq 0}  \frac{B_m \delta^m}{m!^s}.
\end{align*}
Multiplying \eqref{eq:Bm:2} by $\delta^m m!^{-s}$, noting that since $s\geq 1$ we have ${m \choose j}^{1-s} \leq 1$ and ${m \choose j~k}^{1-s} \leq 1$, and recalling the initial datum assumption \eqref{eq:aniso:data}, we arrive at
\begin{align}
\|(\nabla v,\aaa-I)\|_{\delta,s,T} 
&\leq 2 C_1 M + 2 C_1 T^{1/2} (1+T^{1/2})   \sum_{m \geq 0} \sum_{0 < j < m}  \frac{B_j \delta^j}{j!^s}  \frac{B_{m-j} \delta^{m-j}}{(m-j)!^s} \notag\\
&\qquad +  2 C_1 T^{3/2}   \sum_{m \geq 0}  \sum_{0< |(j,k)| < m} \frac{B_j \delta^j}{j!^s} \frac{B_k \delta^k}{k!^s}  \frac{B_{m-j-k} \delta^{m-j-k}}{(m-j-k)!^s} \notag\\
&\leq 2 C_1 M + 2 C_1 T^{1/2} (1+T^{1/2}) \|(\nabla v,\aaa-I)\|_{\delta,s,T}^2 \notag\\
&\qquad \qquad + 2 C_1 T^{3/2} \|(\nabla v,\aaa-I)\|_{\delta,s,T}^3.
\label{eq:Bm:3}
\end{align}
Here we used the discrete Young inequality $\ell^1 \ast \ell^1 \subset \ell^1$.
In order to conclude the proof, we note that the initial values are $\nabla v_0$ obeying \eqref{eq:aniso:data}, and $\aaa_0 =I$.
Thus, at $T=0$ we have
\begin{align}
\|(\nabla v,\aaa-I)\|_{\delta,s,0}  \leq M,
   \label{EQ20}
\end{align}
and in view of \eqref{eq:Bm:3}, if $T$ is taken sufficiently small so that 
\begin{align}
8 C_1^2 T^{1/2} (1+T^{1/2}) M + 32 C_1^3 T^{3/2}  M^2\leq \frac 14,
   \label{EQ27}
\end{align}
we arrive at
\begin{align}
\|(\nabla v,\aaa-I)\|_{\delta,s,T}  \leq 4 C_1 M.
\label{eq:T3}
\end{align}

In summary, we have proven that there exists $T = T(M) > 0$, given by \eqref{eq:T1}, \eqref{eq:T2}, and \eqref{EQ27}, such that 
\begin{align}
&\sum_{m\geq 0} \sup_{t \in [0,T]} \left( \| \partial_1^m \nabla v(t)\|_{H^r}  + \frac{\| \partial_1^m (\aaa(t)-I)\|_{H^r} }{t^{1/2}} \right) \frac{\delta^m}{m!^s} \notag\\
&\qquad \leq C \sum_{m\geq 0} \|\partial_1^m \nabla v_0\|_{H^r} \frac{\delta^m}{m!^s} = C M
   \label{EQ19}
\end{align}
for some constant $C>0$. This concludes the proof of the a~priori estimates needed to establish Theorem~\ref{thm:anisotropic}.

\begin{remark}[\bf Justification of the a priori estimates]\label{R01}
Here we show that by using an approximation argument we may  rigorously justify the inequality \eqref{EQ19}.
Assume that the initial datum
$v_0$ is real-analytic (e.g., a mollified approximation of the original datum) and it satisfies the inequality
\eqref{EQ20}, i.e.,
  \begin{equation}
   \sum_{m=0}^{\infty}
    \Vert \partial_{1}^{m}\nabla v_0\Vert_{H^{r}}
     \frac{\delta^{m}}{m!^{s}}
     \le M
   \label{EQ21}
  \end{equation}
for some $\delta>0$ and $s\ge1$. 
Then by \cite{BBZ,KukavicaVicol11b} we know that the solution is real-analytic on $[0,T_1)$,
where $T_1>0$ (cf.~\eqref{eq:T1}) is the time of existence of the solution $v$ in $H^{r+1}$, which under the assumptions of the theorem
may be taken independently of the mollification parameter, and in particular it is infinite when $d=2$.
Thus $B_{m}(t)<\infty$ for all $t \in [0,T_1)$ and all $m \geq 0$.

Let $m_0\geq 0$ be an arbitrary integer,
and define
$\overline B_m=B_m$ for $m\in\{0,1,\ldots,m_0\}$ and
$\overline B_m=0$ for $m\in\{m_0+1,m_0+2,\ldots\}$.
Similarly, denote by 
$\overline \Omega_m$ the
same type of truncation corresponding to $\Omega_m$, for all integers $m\geq 0$.
Then $\overline B_m$ and $\overline \Omega_m$
satisfy the same recursion relation
\eqref{eq:Bm:1}, i.e.,
  \begin{align*}
   \overline B_m 
    &\leq C_1 \overline\Omega_m 
     + C_1 T^{1/2}  (1+ \overline B_0 + T^{1/2} \overline B_0 
     + T \overline B_0^2) \overline B_m \notag\\
   &\qquad + C_1 T^{1/2} (1+T^{1/2}) \sum_{0 < j<m} {m \choose j}
   \overline B_j \overline B_{m-j} \notag\\
   &\qquad + C_1 T^{3/2}  \sum_{0< |(j,k)| < m} {m\choose j~k} 
       \overline B_j \overline B_k \overline B_{m-j-k}
  \end{align*} 
  for all $m\geq 0$.
Denote
\[
   \overline S_{m_0}(t)
    = \sum_{m=0}^{\infty}
        \frac{\overline B_m(t)\delta^{m}}{m!^{s}}
    = \sum_{m=0}^{m_0}
        \frac{B_m(t)\delta^{m}}{m!^{s}}.
\]
Note that
$\overline S_{m_0}$ is a continuous function of time and
  \begin{equation}
   \overline S_{m_0}(0)
    \le  M
   \label{EQ24}
  \end{equation}
Following the derivation in \eqref{eq:Bm:3}, we then obtain
  \begin{align*}
   \overline S_{m_0}(t)
     & \le
    2 C_1 M + 2 C_1 T^{1/2} (1+T^{1/2}) \overline S_{m_0}^2 
      + 2 C_1 T^{3/2} \overline S_{m_0}^3
  \end{align*}
for all $t\ge0$.
By \eqref{EQ24}
and the continuity of $\overline S_{m_0}(t)$, we get
  \begin{equation}
   \overline S_{m_0}(T)
   \le
   4 C_1 M
   \label{EQ26}
  \end{equation}
provided that $T < T_1$ is chosen to obey \eqref{eq:T1}, \eqref{eq:T2}, and \eqref{EQ27}.
The bound \eqref{EQ26} may be rewritten as
\[
   \sum_{m=0}^{m_0}
     \frac{B_m(t)}{m!^{s}}\delta^{m}
     \le 4 C_1 M
\]
for all $t\in[0,T]$, with $T$ as above.
Finally, since $m_0 \geq 0$ is arbitrary,
from the monotone convergence theorem we obtain
\[
   \sum_{m=0}^{\infty}
     \frac{B_m(t)}{m!^{s}}\delta^{m}
     \le 4 C_1 M
\]
for all $t\in[0,T]$. Passing to zero in the mollification approximation completes the proof.
\end{remark}

\section{Local in time persistence of the Lagrangian Gevrey radius}
\label{sec:isotropic}
In this section we prove Theorem~\ref{thm:isotropic}. For simplicity of the presentation, we give here the proof for $d=3$. 
Fix $s \geq1$ and $\delta >0$ so that $\nabla v_0 \in G_{s,\delta}$ with norm $M$, that is, the quantity
\begin{align*}
\Omega_m:= \sum_{|\alpha|=m} \|\partial^\alpha \nabla v_0\|_{H^r} 
\end{align*}
obeys
\begin{align}
\sum_{m\geq 0} \Omega_m \frac{\delta^m}{m!^s} \leq M
\label{eq:aniso:data:3D}
\end{align}
Fix $T>0$, to be chosen later sufficiently small in terms of $M$, $s$,
and $\delta$. Similarly to the previous section for $m\geq 0$ define 
\begin{align}
V_m &= V_m(T) = \sup_{t\in [0,T]} \sum_{|\alpha|=m} \| \partial^\alpha \nabla v(t)\|_{H^r} \label{eq:Vm:def:3D}, \\
Z_m &= Z_m(T) = \sup_{t\in [0,T]} t^{-1/2} \sum_{|\alpha|=m} \| \partial^\alpha (\aaa(t)-I) \|_{H^r} \label{eq:Zm:def:3D}.
\end{align}

In order to estimate $\nabla v$ and its derivatives, we use the three-dimensional curl-div system \eqref{eq:vort:3D:2} and \eqref{eq:vort:3D:1} to write
\begin{align}
(\curl v)^m &= \eps_{mlk} \partial_l v^k = \omega_0^m + \eps_{ilk} (\delta_{im}-\aaa_i^m) \partial_l v^k + \eps_{mjk}(\delta_{jl} - \aaa_j^l)\partial_l v^k \notag \\
&\qquad \qquad \qquad \quad- \eps_{ijk}(\delta_{im} - \aaa_i^m)(\delta_{jl}-\aaa_j^l) \partial_l v^k \label{eq:3D:curl}\\
\ddiv v &= (\delta_{ik} - \aaa_i^k) \partial_k v^i. \label{eq:3D:div}
\end{align}
From \eqref{eq:3D:div}--\eqref{eq:3D:curl} we conclude that for $\alpha \in {\mathbb N}_0^3$ we have
\begin{align*}
\Vert  \partial^{\alpha}\nabla v\Vert_{H^{r}}
&\leq C\|\partial^{\alpha} \omega_0^m \|_{H^r} + C \|  \partial^{\alpha} ( \eps_{ijk} (\delta_{im} - \aaa_i^m) (\delta_{jl} - \aaa_j^l) \partial_l v^k ) \|_{H^r} \notag\\
&\qquad +C \|  \partial^{\alpha} ( \eps_{mjk}(\delta_{jl}-\aaa_{j}^{l}) \partial_l v^k)\|_{H^r} + C \|  \partial^{\alpha} (\eps_{ijk} (\delta_{im} - \aaa_{i}^m) \partial_j v^k ) \Vert_{H^{r}} \notag\\
&\qquad + C \| \partial^\alpha  ((\delta_{ik} -\aaa_i^k) \partial_k v^i)\Vert_{H^{r}}.
\end{align*}
Summing the above inequality over all multi-indices with $|\alpha|=m$ and taking a supremum over $t \in [0,T]$ we arrive at
\begin{align}
V_m
&\leq C \Omega_m + C T Z_m Z_0 V_0  + C T Z_0^2 V_m  + C T^{1/2} Z_0 V_m + C T^{1/2} Z_m V_0 \notag\\
&\quad  + C T^{1/2} \sum_{0<j<m} \sum_{|\alpha|=m, |\beta|=j, \beta \leq \alpha} {\alpha \choose \beta} \sup_{t\in [0,T]} \left(t^{-1/2}\|\partial^{\beta}(\aaa-I) \|_{H^r}   \|\partial^{\alpha-\beta} \nabla v\|_{H^r} \right) \notag\\
&\quad  + C T \sum_{0 < (j,k) < m} \sum_{|\alpha|=m, |\beta| = j, \beta \leq \alpha, |\gamma|={m-j-k}, \gamma \leq \alpha-\beta} {\alpha \choose \beta~\gamma} \notag\\
&\qquad \qquad \times \sup_{t\in [0,T]} \left(t^{-1/2}\|\partial^{\beta}(\aaa-I) \|_{H^r} t^{-1/2}\|\partial^{\gamma}(\aaa-I) \|_{H^r}   \|\partial^{\alpha-\beta-\gamma} \nabla v\|_{H^r} \right) \notag\\
&\leq C \Omega_m + C T Z_m Z_0 V_0  + C T Z_0^2 V_m  + C T^{1/2} Z_0 V_m + C T^{1/2} Z_m V_0 \notag\\
&\quad + C T^{1/2} \sum_{0<j<m} {m \choose j} Z_j V_{m-j}  + C T \sum_{0 < (j,k) < m} {m \choose j~k} Z_j Z_k V_{m-j-k},
\label{eq:Vm:3D}
\end{align}
for all $m \geq 1$.
In \eqref{eq:Vm:3D} we have used that if $\{ a_\alpha\} , \{ b_\alpha\} , \{c_\alpha\}$ are non-negative multi-indexed sequences, {then}
\begin{align}
\sum_{|\alpha|=m, |\beta|=j, \beta \leq \alpha} {\alpha \choose \beta} a_\beta b_{\alpha -\beta} 
\leq {m\choose j} \left( \sum_{|\beta|=j} a_\beta \right) \left( \sum_{|\gamma|=m-j} b_\gamma \right) 
\label{eq:comb:2D}
\end{align}
and
\begin{align} 
&\sum_{|\alpha|=m, |\beta| = j,\beta \leq \alpha, |\gamma|={m-j-k},\gamma\leq \alpha-\beta} {\alpha \choose \beta~\gamma}  a_\beta b_{\gamma} c_{\alpha -\beta-\gamma}\notag\\
&\qquad \qquad  \qquad\leq {m\choose j~k} \left( \sum_{|\beta|=j} a_\beta \right) \left( \sum_{|\gamma|=k} b_\gamma \right) \left( \sum_{|\alpha|=m-j-k} c_\alpha \right).
\label{eq:comb:3D}
\end{align}
These inequalities follow e.g.~from~\cite[Lemma
4.2]{KukavicaVicol11a} and \cite[Lemma A.1]{KukavicaVicol11b} and the
fact that ${\alpha \choose \beta} \leq {|\alpha| \choose |\beta|}$. 
Indeed, for \eqref{eq:comb:2D} (the proof of \eqref{eq:comb:3D} being
analogous), we have by using the substitution
$\gamma=\alpha-\beta$
  \begin{align}
   &\sum_{|\alpha|=m, |\beta|=j, \beta \leq \alpha} {\alpha \choose \beta} a_\beta b_{\alpha -\beta} 
   =
   \sum_{|\beta|=j}
    \sum_{|\gamma|=m-j}
     {\beta+\gamma \choose \beta} a_\beta b_{\gamma}
   \nonumber\\&\qquad
     \leq {m\choose j} 
     \sum_{|\beta|=j}
      \sum_{|\gamma|=m-j}
        a_\beta b_{\gamma}
    .
   \label{EQ30}
\end{align}
Note that when $m=0$, the bound \eqref{eq:Vm:3D} reads as
\begin{align} 
V_0 &\leq C_0 \Omega_0 + C_0 T^{1/2}(T^{1/2}  Z_0^2 + Z_0) V_0 
\label{eq:V0:3D}
\end{align}
for some constant $C_0>0$.

As in the two-dimensional case, in order to bound $Z_m$ we appeal to the integral formula for $Y(t) - I$, namely \eqref{eq:a:integrated}. We apply $\partial^\alpha$ to identity \eqref{eq:a:integrated}, sum over all multi-indices with $|\alpha|= m$, divide the resulting inequality by $t^{1/2}$ and take a supremum over $t \in [0,T]$. By appealing to \eqref{eq:comb:2D} and \eqref{eq:comb:3D}, similarly to \eqref{eq:Zm} we obtain
\begin{align}
Z_m 
&\leq C T^{1/2} (T Z_0^2 V_m + T Z_m Z_0 V_0 + T^{1/2} Z_0 V_m + T^{1/2} Z_m V_0 + V_m)  \notag\\
&\qquad + C T^{3/2} \sum_{0 < |(j,k)| < m} {m\choose j~k} Z_j Z_k V_{m-j-k} + C T  \sum_{j=1}^{m-1}  {m \choose j} Z_j V_{m-j}
\label{eq:Zm:3D}
\end{align}
when $m\geq 1$, and 
\begin{align} 
Z_0 \leq C_0 T^{1/2} (1+T^{1/2} Z_0)^2 V_0 
\label{eq:Z0:3D}
\end{align}
for $m=0$.

Once the recursive bounds \eqref{eq:Vm:3D}--\eqref{eq:V0:3D} and \eqref{eq:Zm:3D}--\eqref{eq:Z0:3D} have been established, we combine them with the initial datum assumption \eqref{eq:aniso:data:3D}, and as in Section~\ref{sec:anisotropic} obtain that there exists $T = T(M) > 0$ such that 
\begin{align*}
&\sum_{\alpha \geq 0} \sup_{t \in [0,T]} \left( \| \partial^\alpha \nabla v(t)\|_{H^r} +\frac{\| \partial^\alpha (\aaa(t)-I)\|_{H^r} }{t^{1/2}}   \right)\frac{\delta^{|\alpha|}}{|\alpha|!^s}  \notag\\ 
&\qquad \leq C \sum_{\alpha\geq 0} \|\partial^\alpha \nabla v_0\|_{H^r} \frac{\delta^{|\alpha|}}{|\alpha|!^s} = M,
\end{align*}
for some constant $C>0$. This concludes the proof of the a priori estimates needed to establish Theorem~\ref{thm:isotropic}.

\section{Example of Eulerian ill-posedness in the analytic class $G_{1,\delta}$} \label{sec:G1delta} 
In this section we prove Theorem~\ref{thm:example}. 
The idea of the proof is similar to the example given earlier in
Remark~\ref{rem:analyticity}, but addresses the fact that functions
whose holomorphic extension have a simple pole at $\pm i \delta$ do
not lie in $G_{1,\delta}$ (a fact encoded in the {\em sum over $m$}, as opposed to a supremum over $m$, defining our real-analytic norm, cf.~\eqref{eq:G:s:delta} and \eqref{eq:G:delta}). To address this issue we integrate such a real-valued function four times, so that the holomorphic extension to the strip of radius $\delta$ (where $\delta=1$) around the real-axis is also a $C^{2}$ function up to the boundary of this strip (cf.~\eqref{eq:big:H:def}). The proof then proceeds by cutting off in a Gaussian way at infinity (cf.~\eqref{eq:big:phi:def}), which is compatible with real-analyticity, and then periodize the resulting function so that we are dealing with a finite energy function (cf.~\eqref{eq:phi:def}). Verifying that the resulting function $\phi$ yields the necessary counterexample to prove the theorem follows then from a direct but slightly technical calculation.

Let $f,g$ be two $2\pi$-periodic functions. Recall (cf.~\cite{DiPernaMajda87a,BardosTiti10}) that the  function defined by
\[
u(x_1,x_2,x_3,t) = (f(x_2), 0 , g(x_1 - t f(x_2)))
\]
is an exact solution of the Euler equations posed on $\TT^3$, 
where $\TT = [-\pi,\pi]$ 
with the 
initial datum
\[
u_0(x_1,x_2,x_3) = (f(x_2),0,g(x_1)).
\]
Also, for a $2\pi$-periodic function $\phi$ and for $\delta > 0$ by definition we have that 
\begin{align}
\| \phi \|_{G_{1,\delta}} = \sum_{m = 0}^{\infty} \left( \sum_{|\alpha| = m} \| \partial^\alpha \phi \|_{H^2(\TT^3)} \right) \frac{\delta^m}{m!}.
\label{eq:G:delta}
\end{align}
Note that $H^2(\TT^3) \subset C^{0}(\TT^3)$ in view of the Sobolev embedding. Without loss of generality we fix $\delta = 1$ throughout this section.

We start with a few considerations on the real line $\RR$. For a function $F \in L^1(\RR)$ we normalize the Fourier transform as
\[
\hat F(\xi) = \frac{1}{\sqrt{2 \pi}} \int_{\RR} F(x) e^{- i x \xi} dx.
\]
Consider the two decaying real-analytic functions
\[   
h_1(x) = \sqrt{\frac{2}{\pi}} \frac{1}{1+ x^2}
\]
and
\[
   h_2(x) = \frac{1}{\sqrt{2}} \exp\left(- \frac{x^2}{4}\right).
\]
These functions have explicit Fourier transforms that are given by
  \begin{equation}
   \hat h_1(\xi) =  \exp(- |\xi|)
   \label{EQ06}
  \end{equation}
and
  \begin{equation}
   \hat h_2(\xi) = \exp(- |\xi|^2).
   \label{EQ07}
  \end{equation}
Define 
\begin{align}
h(x) = h_1(x) - \left(1 - (-\Delta)^{1/2} +\frac{3}{2} (-\Delta) - \frac{7}{6} (-\Delta)^{3/2}\right) h_2(x)
\label{eq:h:def:def}
\end{align}
In view of the above formulae we have that
\[
\hat h(\xi) =   \exp(-|\xi|) - \left(1 - |\xi| +\frac{3}{2}|\xi|^2 - \frac{7}{6} |\xi|^3 \right) \exp(-|\xi|^2).
\]
Note that
\[
\hat h(\xi) = \frac{25 |\xi|^4}{24} + O(|\xi|^5) \qquad \mbox{as} \qquad |\xi| \to 0
\]
and
\[
\hat h(\xi) = \exp(-|\xi|) + O\left(\exp\left(-\frac{|\xi|^2}{2}\right)\right) \qquad \mbox{as} \qquad |\xi| \to \infty.
\]
Lastly, we define
\begin{align}
H(x) = \int_0^x \int_0^{x_1} \int_0^{x_2} \int_{0}^{x_3} h(x_4) dx_4\,d x_3 \, dx_2 \,dx_1,
\label{eq:big:H:def}
\end{align}
so that 
\begin{align}
\frac{d^4}{dx^4} H(x) = h(x).
\label{eq:H:d:4}
\end{align}
By taking the Fourier transform of the above equation we arrive at
\begin{align}
\hat{H}(\xi) = \frac{\hat h(\xi)}{(i\xi)^4} = \frac{1}{|\xi|^4}
   \left(\exp(-|\xi|) - \left(1- |\xi| +\frac{3}{2}|\xi|^2 - \frac{7}{6} |\xi|^3\right) \exp(-|\xi|^2)\right).
\label{eq:H:hat}
\end{align}
Clearly,
\begin{align}
\sup_{|\xi|\leq 1} |\hat H(\xi)| + \sup_{|\xi|\geq 1} \left( |\xi|^4 \exp(|\xi|) |\hat H(\xi)| \right) \leq C_0
\label{eq:H:bound}
\end{align}
for some constant $C_0>0$.
The function $H$ however is not in $L^1$ since it grows 
as $|x| \to \infty$, and the above computations 
are formal.
To fix this issue, we set
\begin{align}
\Phi(x) = \exp\left( - \frac{x^2}{2} \right) H(x).
\label{eq:big:phi:def}
\end{align}
This function is smooth, and decays as $|x| \to \infty$. Moreover, in
view of \eqref{eq:H:hat} and using the explicit Fourier transform of the Gaussian, we have
\begin{align*}
\hat \Phi(\xi) &= \int_{\RR} \exp\left( -\frac{(\xi-\eta)^2}{2} \right) \hat H(\eta) d\eta  \notag\\
&= \int_{\RR} \exp\left( -\frac{(\xi-\eta)^2}{2} \right) \frac{1}{|\eta|^4} \left(\exp(-|\eta|) - \left(1- |\eta| +\frac{3}{2}|\eta|^2 - \frac{7}{6} |\eta|^3\right) \exp(-|\eta|^2) \right)d\eta.
\end{align*}
We claim that 
\begin{align}
\sup_{|\xi|\leq 1} |\hat \Phi(\xi)| + \sup_{|\xi|\geq 1} \left( |\xi|^4 \exp(|\xi|) |\hat \Phi(\xi)| \right) \leq C_1
\label{eq:phi:bound}
\end{align}
for some universal constant $C_1>0$.
In order to check whether \eqref{eq:phi:bound} holds, we write 
\begin{align*}
|\xi|^4 \exp(|\xi|) \hat \Phi(\xi) &= - \int_{\RR} \exp\left( -\frac{(\xi-\eta)^2}{2} \right)\exp(|\xi|-|\eta|) |\xi|^4  \notag\\
&\qquad \qquad \times \frac{ 1  - (1- |\eta| +\frac{3}{2}|\eta|^2 - \frac{7}{6} |\eta|^3)  \exp(-|\eta|^2 + |\eta|)}{|\eta|^4} d\eta,
\end{align*}
decompose the above integral in the regions
\[
\left\{ |\eta|\leq \frac 14 \right\}, \quad \left\{ \frac 14 \leq |\eta|\leq |\xi|^{3/4}\right\}, \quad \left\{|\xi|^{3/4} \leq |\eta|\leq |\xi| \right\}, \quad \left\{ |\eta|\geq |\xi| \right\}
\]
and use both the decay resulting from the Gaussian factor and the
decay coming from \eqref{eq:H:bound}.

A useful observation that shall be needed below is that we have 
\[
\bigl\| |\xi|^k \exp(-|\xi|) \bigr\|_{L^2(\RR)} = \frac{ \sqrt{(2k)!}}{2^k} 
\]
which by Stirling's estimate
  \begin{equation}
   (2\pi)^{1/2}
   n^{n+1/2}e^{-n}
   \le
   n!
   \le
   e n^{n+1/2}
   e^{-n}
   \comma 
   n\in{\mathbb N}
   \label{EQ01}
  \end{equation}
yields
\begin{align}
\frac{1}{k!} \bigl\| |\xi|^k \exp(-|k|) \bigr\|_{L^2(\RR)} 
\le \frac{1}{k^{1/4}}.
\label{eq:Stirling}
\end{align}

Now, we proceed to construct a periodic function with a finite
$G_{1,1}$ norm.
First, we build a $2\pi$-periodic function $\phi$ by using the Poisson summation applied to the function $\Phi$. More precisely, let
\begin{align}
\phi(x) = \sum_{m = -\infty}^{\infty} \Phi(x - 2 m \pi).
\label{eq:phi:def}
\end{align}
Clearly $\phi$ is periodic, and its Fourier series coefficients obey
\begin{equation}
\hat{\phi}(k) = \frac{1}{\sqrt{2\pi}} \hat\Phi(k),
\end{equation}
for all $k \in \ZZ$. 

Therefore, using estimates \eqref{eq:phi:bound} and \eqref{eq:Stirling},
with the Poisson summation formula, we have that 
\begin{align*}
\|\phi\|_{G_{1,1}} 
&= 
\sum_{n \geq 0} \left\| \frac{d^n}{dx^n} \phi\right\|_{H^2(\TT)} \frac{1^n}{n!} \notag\\
&\leq C \|\phi\|_{H^6(\TT)} + C \sum_{n \geq 5} \left( \left\| \frac{d^n}{dx^n} \phi\right\|_{L^2(\TT)} + \left\| \frac{d^{n+2}}{dx^{n+2}} \phi\right\|_{L^2(\TT)} \right) \frac{1}{n!}  \notag\\
&\leq C \|\phi\|_{H^6(\TT)} +  C \sum_{n \geq 5} \left( \| |k|^n \hat\phi(k) \|_{L^2(\ZZ)} + \| |k|^{n+2} \hat\phi(k) \|_{L^2(\ZZ)} \right) \frac{1}{n!}  \notag\\
&\leq C \|\phi\|_{H^6(\TT)} + C \sum_{n \geq 5} \left( \| |\xi|^n \hat\Phi(\xi) \|_{L^2(|\xi|\geq 1)} + \| |\xi|^{n+2} \hat\Phi(\xi) \|_{L^2(|\xi|\geq 1)} \right) \frac{1}{n!}
\end{align*}
and thus
\begin{align}
\|\phi\|_{G_{1,1}} 
&\leq
 C \|\phi\|_{H^6(\TT)} 
\notag\\&\qquad
+ C \sum_{n \geq 5}  \biggl(  \| |\xi|^{n-4} \exp(-|\xi|) \|_{L^2(|\xi|\geq 1)} 
+ \| |\xi|^{n-2} \exp(-|\xi|) \|_{L^2(|\xi|\geq 1)} \biggr)\frac{1}{n!}  \notag\\
&\leq C \|\phi\|_{H^6(\TT)} + C \sum_{n \geq 5}  \left( \frac{(n-4)!}{(n-4)^{1/4}} + \frac{(n-2)!}{(n-2)^{1/4}} \right)\frac{1}{n!}  \notag\\
&\leq C \|\phi\|_{H^6(\TT)}  + C \sum_{n \geq 5}\frac {1}{n^{9/4}} \notag\\
&\leq C_\phi < \infty.
\label{eq:phi:G11}
\end{align}
Note that $\|\phi\|_{G_{1,\delta}} = \infty$ for any  analyticity radius $\delta>1$, since 
\[
\sum_{n\geq 5} \frac{(n-2)!}{(n-2)^{1/4}} \frac{\delta^n}{n!} = \infty
\]
whenever $\delta>1$, and the estimate in \eqref{eq:phi:G11} may also be turned into lower bounds.


{\begin{proof}[Proof of Theorem~\ref{thm:example}]
Consider
\begin{align}
g(x)= \phi(x)
\label{eq:g:def}
\end{align}
where $\phi$ is as given in \eqref{eq:phi:def}, 
and define
\begin{align}
f(x) = \sin x
\label{eq:f:def}
\end{align}
Since $f$ is entire, we have that $\|f\|_{G_{1,\delta}} <\infty$ for any $\delta >0 $.
With the definitions of $f$ and $g$ above, it follows from \eqref{eq:phi:G11} that 
\[
\|u_0\|_{G_{1,1}} < \infty.
\]
Note that in view of the periodicity in $x_1$ and $x_2$, the functions $f(x_2)$ and $g(x_1)$ have finite energy (i.e., $H^2(\TT^3)$ becomes $H^2(\TT)$, up to a multiplicative constant), and the multi-index summation in \eqref{eq:G:delta} becomes a simple sum over $n \geq 0$. Thus \eqref{eq:IC} is established.

In order to establish \eqref{eq:t>0}, we assume,
for the sake of obtaining a contradiction,
that for some $t \in (0,1/10]$ we have $\|u(t)\|_{G_{1,1}}< \infty$. We fix this value of $t \in (0,1/10]$ throughout this proof. 

Consider the function
\begin{align}
\psi(x_1,x_2) := \partial_{x_1}^3 u_3(x_1,x_2,x_3,t) = g'''(x_1 - t f(x_2)).
\label{eq:psi:3:def}
\end{align}
The inequality $\|u(t)\|_{G_{1,1}} <\infty $ implies 
\[
\sum_{\alpha \geq 0} \|\partial^\alpha \psi\|_{H^2} \frac{1}{(|\alpha|+3)!} <\infty.
\]
It follows that for any $R \in (0,1)$, the joint in $(x_1,x_2)$ power 
series of $\psi$ at the origin
\begin{align}
\psi(x_1,x_2) = \sum_{m,n\geq 0} a_{m,n} x_1^m x_2^n
\label{eq:psi:power:series}
\end{align}
converges absolutely in the closed square of side length $R$ at the origin 
\[
{\mathcal C}_R = \{ (x_1,x_2) \colon |x_1|\leq R, |x_2|\leq R\}
\]
and defines a real-analytic function of two variables in this square.
Thus, we may consider the complex extension
\[
\psi(z_1,z_2) = \psi(x_1+iy_1,x_2+iy_2) = \sum_{m,n\geq 0} a_{m,n} z_1^m z_2^n
\]
which converges absolutely when $|z_1|\leq R$ and $|z_2| \leq R$.
Fix 
  \begin{equation}
   R_t = 1- \frac{3 t }{4}
   \label{EQ08}
  \end{equation}
which clearly belongs to $(0,1)$, and is thus an allowable choice for $R$. Also, fix
\[
x_1=0 \qquad \mbox{and} \qquad z_2 = 0 + i \log 2.
\]
Since
$t\in(0,1/10)$, we have  $|z_2| = \log 2 < R_t$, so that by the above consideration, 
\begin{align}
\lim_{y_2 \to - R_t^-} |\psi(i y_2, i \log 2)| < \infty.
\label{eq:bad:assumption}
\end{align}
In order to complete the proof by contradiction, 
we shall next show that in fact \eqref{eq:bad:assumption} is false, and in fact we have
\begin{align}
\lim_{y_2 \to - R_t^-} |\psi(i y_2, i \log 2)| = \infty.
\label{eq:BAD:TODO}
\end{align}
The remainder of this proof is devoted to establishing \eqref{eq:BAD:TODO}.

First observe that $\sin(i \log 2) = 3 i /4$, and thus
\[
\psi(i y_2, i \log 2) = \phi'''(i (y_2 - 3t/4)).
\]
Next, note that by the definition of $R_t$, \eqref{EQ08}, we have
\[
y_2 - \frac{3t}{4} \to -1^+ \qquad \mbox{as} \qquad y_2 \to - R_t^-.
\]
Thus, proving \eqref{eq:BAD:TODO} amounts to showing that 
\begin{align}
\lim_{y \to -1^+} |\phi'''(i y)| = \infty
\label{eq:BAD}
\end{align}
which is what we establish below.
In view of \eqref{eq:big:phi:def}, \eqref{eq:phi:def}, and the Leibniz rule, we have that 
\begin{align}
\phi'''(z) 
= \Phi'''(z) + \sum_{m\in \ZZ\setminus\{0\}} \Phi'''(z - 2 m \pi)
\label{eq:phi:d:3}
\end{align}
and
\begin{align}
\Phi'''(z) = \exp\left(-\frac{z^2}{2} \right) 
  \bigl(H'''(z) - 3 z H''(z) + 3 (z^2-1) H'(z) + z(3-z^2) H(z) \bigr)
\label{eq:big:phi:d:3}
\end{align}
for any complex number $z$ with $|z|<1$.
Next, note that by \eqref{eq:h:def:def} and \eqref{eq:H:d:4}, we have 
\[
H^{(iv)}(z) = h_1(z) + \EE_0(z)
\]
where 
  \begin{equation}
    \EE_0(z) = \left(1 - (-\Delta)^{1/2} +\frac{3}{2} (-\Delta) - \frac{7}{6} (-\Delta)^{3/2}\right) h_2(z)
   \label{EQ10}
  \end{equation}
is an entire function (since its Fourier coefficients are given by a polynomial times a decaying Gaussian). Moreover, letting $\EE_1(z) = \int_0^z \EE_0(w_1) dw_1$, $\EE_2(z) = \int_0^z \!\! \int_0^{w_1} \EE_0(w_2) dw_2\, dw_1$, $\EE_3(z) = \int_0^z \!\! \int_0^{w_1} \!\! \int_0^{w_2} \EE_0(w_3) dw_3\, dw_2\, dw_1$, and $\EE_4(z) = \int_0^z \!\! \int_0^{w_1}\!\! \int_0^{w_2} \!\!\int_0^{w_3} \EE_0(w_4) dw_4\, dw_3\, dw_2\, dw_1$, 
we immediately obtain that
\begin{align}
\EE(z) = \EE_1(z) - 3 z \EE_2(z) + 3 (z^2-1) \EE_3(z) + z(3-z^2) \EE_4(z) 
\label{eq:big:E:def}
\end{align}
is also an entire function. On the other hand, we may explicitly 
compute the integrals of $h$ as
\begin{align*}
\HH_1(z) &= \int_0^z h_1(w_1) dw_1 = \sqrt{\frac{2}{\pi}} \arctan z \\
\HH_2(z) &=\int_0^z \!\! \int_0^{w_1} h_1(w_2)  dw_2\, dw_1 = \sqrt{\frac{2}{\pi}}  \left(z \arctan z - \frac 12 \log(1 + z^2) \right) \notag \\
\HH_3(z) &=\int_0^z\!\! \int_0^{w_1}\!\! \int_0^{w_2} h_1(w_3) dw_3\,  dw_2\, dw_1 = \frac 12 \sqrt{\frac{2}{\pi}}  \left(z + (z^2-1) \arctan z - z \log(1 + z^2) \right)\notag \\
\HH_4(z) &=\int_0^z\!\! \int_0^{w_1}\!\! \int_0^{w_2} \!\! \int_0^{w_3} h_1(w_4) dw_4\, dw_3\,  dw_2\, dw_1 \notag \\
&= \frac{1}{12} \sqrt{\frac{2}{\pi}}  \left(5z^2 + 2z (z^2-3) \arctan z - (3 z^2-1) \log(1 + z^2) \right)
\end{align*}
which implies that
\begin{align}
\HH(z) &:=\HH_1(z) - 3 z \HH_2(z) + 3 (z^2-1) \HH_3(z) + z(3-z^2) \HH_4(z) \notag \\
&=  \frac{1}{12} \sqrt{\frac{2}{\pi}} \Big(z (-18 + 33 z^2 - 5 z^4) \notag \\
&\qquad \quad + 2 (15 - 45 z^2 + 15 z^4 - z^6) \arctan z +  z (39 - 28 z^2 + 3 z^4) \log(1+z^2) \Big).
\label{eq:big:HH:def}
\end{align}
In summary, with the definition of $\EE$ in \eqref{eq:big:E:def} and of $\HH$ in \eqref{eq:big:HH:def}, we have that 
\begin{align}
\Phi'''(z) = \exp\left(-\frac{z^2}{2} \right) (\HH(z) + \EE(z)).
\end{align}
Letting $z = i y$, and using that $\arctan(i y) = i {\rm arctan} y$, 
we arrive at 
\begin{align}
\Phi'''(iy) = \exp\left(\frac{y^2}{2}\right)(\HH(iy) + \EE(iy)).
\end{align}
Since $\EE$ is an entire function, we have that $\sup_{y \in [-1,0]} |\EE(iy)| \leq C < \infty$. 
Writing
\begin{align}
\HH(iy) &=  i \sqrt{\frac{2}{\pi}} y^2 \arctanh y + \frac{i}{12} \sqrt{\frac{2}{\pi}} y  (-18 - 33 y^2 - 5 y^4) \notag \\
&\quad + \frac{i}{12} \sqrt{\frac{2}{\pi}} (39 + 28 y^2 + 3 y^4) \Big( 2  \arctanh y +  y  \log(1-y^2) \Big) \notag\\
&\quad + \frac{i}{6} \sqrt{\frac{2}{\pi}} (- 24 + 11 y^2 + 12 y^4 + y^6) \arctanh y 
\label{eq:HH:iy}
\end{align}
and observing
$ \lim_{y \to -1^+} \left( 2 \arctanh y + y \log(1-y^2) \right) = - \log 4$ and $ \lim_{y \to -1^+} (y+1)  \arctanh y  =0$, 
we arrive at
\begin{align}
\lim_{y\to -1^+} |\HH(iy)| = \infty
\label{eq:HH:-1:blowup}
\end{align}
since $\arctanh$ has a logarithmic singularity at $y = -1$. Combined with the above, it follows from \eqref{eq:HH:-1:blowup} that 
\begin{align}
\lim_{y\to -1^+} |\Phi'''(iy)| = \infty
\end{align}
which in turn shall imply that \eqref{eq:BAD} holds. 

Indeed, the only remaining part of the proof is to show that 
\[
\lim_{y\to -1^+}  \sum_{m\in \ZZ\setminus\{0\}} \left|\Phi'''(iy - 2 m \pi) \right| < \infty.
\]
The above holds since for each $m\neq 0$ we have that
\[
|\HH(i y - 2m\pi)| + |\EE(iy - 2 m \pi)| 
\leq 
P(m)
\]
uniformly for  $|y| \in [1/2,1]$,
where $P$ is a polynomial, and since 
\[
\left| \exp\left( - \frac{(iy - 2 m \pi)^2}{2} \right) \right| \leq \exp\left(\frac 12 - 2 m^2 \pi^2\right)
\]
which makes the sum over $m\neq 0$ finite.
In order to obtain the first bound,
we use \eqref{eq:HH:iy} and the formula
  \begin{equation*}
   \arctan z 
   = \frac12 i \bigl(
                \log (1-iz) - \log (1+iz)
               \bigr)
  \end{equation*}
where the complex domains of the above logarithms
are cut on $[0,\infty)$ and $(-\infty,0]$
respectively. 
\end{proof}


\section*{Acknowledgments} 
The authors are thankful to Alexander Shnirelman for helpful suggestions.
The work of PC was supported in part by the NSF grants DMS-1209394 and DMS-1265132, 
IK was supported in part by the NSF grant DMS-1311943,
while the work of VV was supported in part by the NSF grants DMS-1348193, DMS-1514771, and by an Alfred P.~Sloan Research Fellowship.


\end{document}